\documentclass[leqno,12pt]{amsart}
\usepackage{setspace}
\usepackage[english]{babel}
\usepackage{bussproofs}
\usepackage{csquotes}
\usepackage{dirtytalk}
\usepackage{mathtools}
\usepackage{mathrsfs}
\usepackage{latexsym}
\usepackage{enumitem}
\usepackage{amsthm}
\usepackage{amssymb}
\DeclareMathAlphabet{\mathpzc}{OT1}{pzc}{m}{it}
\usepackage[latin1]{inputenc}
\usepackage{afterpage}

\usepackage{bbm}
\usepackage[rgb,dvipsnames]{xcolor}
\usepackage{tikz} % for diagrams and drawing
\usetikzlibrary{decorations.text}
\usetikzlibrary{arrows,arrows.meta,hobby}
\usetikzlibrary{shapes.misc, positioning}

\usepackage{tikz-cd}

\usetikzlibrary{cd}

\newtheorem{theorem}{Theorem}[section]
\newtheorem{maintheorem}{Main Theorem}[section]

\newtheorem{lemma}[theorem]{Lemma}

\newtheorem{corollary}[theorem]{Corollary}

\newtheorem{fact}[theorem]{Fact}
\newtheorem{claim}[theorem]{Claim}
\theoremstyle{definition}
\newtheorem{definition}[theorem]{Definition}

\theoremstyle{remark}

\newtheorem{notation}{Notation}
\newtheorem{question}{Question}

\def\forces{\Vdash}

\def\ZFC{\mathsf{ZFC}}

\def\MA{\mathsf{MA}}

\def\baire{\omega^\omega}

\def\CH {\mathsf{CH}}
\def\Q{\mathbb Q}
\def\P{\mathbb P}

\def\PT{\mathbb{PT}}

\def\mfa{\mathfrak{a}}

\def\mfi{\mathfrak{i}}

\def\N{\mathcal N}
\def\cof{{\rm cof}}

\def\non{{\rm non}}

\def\mfu{\mathfrak{u}}

\usepackage[margin=1.3in]{geometry}

\begin{document}

\title{Selective Independence and $h$-Perfect Tree Forcing Notions}

\author[Switzer]{Corey Bacal Switzer}
\address[C.~B.~Switzer]{Institut f\"{u}r Mathematik, Kurt G\"odel Research Center, Universit\"{a}t Wien, Kolingasse 14-16, 1090 Wien, AUSTRIA}
\email{corey.bacal.switzer@univie.ac.at}

\thanks{\emph{Acknowledgements:} The author would like to thank the
Austrian Science Fund (FWF) for the generous support through grant number Y1012-N35. This paper was written for the RIMS K\^{o}ky\^{u}roku volume on RIMS set theory workshop 2021. The author thanks the organizers of that conference.}

\date{}

\maketitle

\begin{abstract}
Generalizing the proof for Sacks forcing, we show that the $h$-perfect tree forcing notions introduced by Goldstern, Judah and Shelah preserve selective independent families even when iterated. As a result we obtain new proofs of the consistency of $\mfi = \mfu < \non (\N) = \cof (\N)$ and $\mfi < \mfu = \non (\N) = \cof( \N)$ as well as some related results. 
\end{abstract}

\section{Introduction}
A family $\mathcal I \subseteq [\omega]^\omega$ is said to be {\em independent} if for all finite subsets $X_0, ..., X_{n-1} \in \mathcal I$ and all $g:n \to 2$ we have that $X^{g(0)}_0 \cap ... \cap X^{g(n-1)}_{n-1}$ is infinite where $X_i^{g(i)}$ means $X_i$ if $g(i) = 1$ and $\omega \setminus X_i$ if $g(i) = 0$. Such a set is said to be {\em maximal} if it is not properly contained in any other independent set. The {\em independence number} $\mfi$ is the least size of a maximal independent family. Despite being one of the classical cardinal characteristics, $\mfi$ is notoriously difficult to manipulate. Indeed many relatively simple open questions remain surrounding $\mfi$, most notably the consistency of $\mfi < \mfa$ where $\mfa$ is the almost disjointness number\footnote{See the appendix of \cite{CFGS21} for a discussion of this problem.}. Part of the issue is that $\mfi$ has no known upper bound, besides the trivial $2^{\aleph_0}$ while it has several lower bounds thus preserving $\mfi$ small requires preserving the smallness of several other cardinal characteristics simultaneously.

One of the first breakthroughs in studying $\mfi$ came in \cite{Sh92} where the consistency of $\mfi < \mfu$ was established\footnote{The cardinal $\mfu$, the {\em ultrafilter number} is the least size of an ultrafilter base on $\omega$.}. There, a special independent family, now known as a {\em selective} independent family was constructed under $\CH$ and it was shown that under somewhat delicate conditions such a family's maximality could be preserved over a countable support iteration of length $\omega_2$ of certain proper forcing notions. Since then selective independence has become one of the main tools in providing models of $\mfi < 2^{\aleph_0}$ with interesting properties. See e.g \cite{DefMIF, FM19, CFGS21}. In particular selective independent families are Sacks indestructible.

All the published examples in the literature\footnote{At least all examples the author is aware of.} of countable support iterations of proper forcing notions which are shown to preserve selective independent families are such that the iterands all have the Sacks property. The Sacks property is used throughout these proofs\footnote{Though note that the Sacks property is not enough to ensure that selective independence is preserved: Silver forcing has the Sacks property but will kill the maximality of any ground model independent family and, if iterated $\omega_2$ many times will result in a model of $\mfi = 2^{\aleph_0} > \aleph_1$.}. However the Sacks property is not needed and is overkill. In this article we show that the $h$-perfect tree forcing notions\footnote{See Definition \ref{hperfecttreedef} for the definition of these posets.} introduced by Goldstern, Judah and Shelah in \cite{SMZnoCohen} also preserve selective independence, though they may, and often do depending on $h$, fail to have the Sacks property. 
\begin{maintheorem}[$\CH$]
\begin{enumerate}
\item
Let $h:\omega \to \omega$ be any function so that for all $n < \omega$ we have $1 < h(n) < \omega$ then the $h$-perfect tree forcing, $\PT_h$ preserves any ground model selective independent family.
\item
Let $\delta$ be an ordinal and $\langle \P_\alpha, \dot{\Q}_\alpha \; | \; \alpha < \delta\rangle$ be a countable support iteration so that for all $\alpha < \delta$ we have \begin{center}$\forces_\alpha$``\, $\dot{\Q}_\alpha$ is the $h$-perfect tree forcing for some $h \in \baire$ with $1 < h(n) < \omega$ for all $n < \omega$".\end{center} 
\noindent Then $\P_\delta$ preserves all ground model selective independent families.
\end{enumerate}
\label{mainthm1}
\end{maintheorem}

This allows us to show that in the models obtained by iterating such forcing notions there is a selective independent family of size $\aleph_1$ and, in particular, $\mfi = \aleph_1$. As a result we obtain the following consistency results.

\begin{maintheorem}
The following are consistent.
\begin{enumerate}
\item
$\mfi = \mfu < \non(\N) = \cof(\N) = 2^{\aleph_0}$
\item
$\mfi < \mfu = \non(\N) = \cof(\N) = 2^{\aleph_0}$
\end{enumerate}
\label{mainthm2}
\end{maintheorem}

Finally riffing off work of Brendle, Fischer and Khomskii \cite{DefMIF}, Schilhan \cite{Schilhanultrafilter} and Bergfalk, Fischer and the author \cite{BFS21} we can obtain that the cardinal characteristic inequalities above are consistent with $\Pi^1_1$-definable witnesses and a $\Delta^1_3$-definable well order of the reals.

\begin{maintheorem}
The cardinal characteristic inequalities featured in Main Theorem \ref{mainthm2} are consistent with a $\Delta^1_3$ well-order of the reals, a $\Pi^1_1$ witness to $\mfi = \aleph_1$ and, in the case of the first inequality, a $\Pi^1_1$ witness to $\mfu = \aleph_1$.
\label{mainthm3}
\end{maintheorem} 

The rest of this paper is organized as follows. In the next section we review the basics of maximal independent families and selective independence. In the following section we introduce the $h$-perfect tree forcing of \cite{SMZnoCohen} and prove Main Theorem \ref{mainthm1}. In Section 4 we move on to applications and prove in particular Main Theorems \ref{mainthm2} and \ref{mainthm3}. We also discuss the relation between independent families and strong measure zero sets. Section 5 concludes with a discussion and some open questions. Throughout most of our terminology is standard and conforms e.g. to that of \cite{JechST} and \cite{Hal17}. For combinatorial cardinal characteristics of the continuum we follow \cite{BlassHB}.

Let us finally stress that most of the results in this paper, in particular Main Theorem \ref{mainthm1} are probably not new and indeed were suggested to the author by both J\"{o}rg Brendle and Vera Fischer\footnote{Private communication.}. However, they do not seem to ever have been written down, at least not explicitly and this seemed worth while to do. In particular, while the consistency of $\mfi < \non(\N)$ was shown in \cite[Theorem 3.8]{BHHH04}, the proof uses a short finite support iteration of ccc forcing notions over a model of $\MA$ (a ``dual iteration") and hence is very different than the model constructed here. Moreover in the model in \cite{BHHH04} we do not know the value of $\mfu$. 

\bigskip

\noindent {\em Acknowledgments}. The author thanks J\"{o}rg Brendle for pointing out \cite{SMZnoCohen} to him and suggesting that $h$-perfect forcing may preserve selective independent families. The author thanks Vera Fischer for many very helpful conversations on this material and sharing her wealth of knowledge on selective independence and the cardinal $\mfi$.

\section{Selective Independence}
In this section we introduce the notion of a {\em selective independent family}. The reader familiar with this idea, for example as presented in \cite{DefMIF} or \cite{FM19}, can comfortably skip this section as nothing is new. Selective independent families were introduced implicitly in Shelah's proof of the consistency of $\mfi < \mfu$ in \cite{Sh92}. To facilitate the discussion we utilize the following notation.

\begin{notation} For $\mathcal{I}\subseteq [\omega]^{\omega}$,\begin{itemize}
\item let $\mathrm{FF}(\mathcal{I})$ denote the set of finite partial functions $g$ from $\mathcal{I}$ to $\{0,1\}$, and
\item for $g\in\mathrm{FF}(\mathcal{I})$  write $\mathcal{I}^g$ for $$\bigcap\{A\mid A\in\mathrm{dom}(g)\textnormal{ and }g(A)=1\}\cap\bigcap\{\omega\backslash A\mid A\in\mathrm{dom}(g)\textnormal{ and }g(A)=0\}.$$
\end{itemize}
\end{notation}

In this notation, a family $\mathcal{I}\subseteq [\omega]^\omega$ is \emph{independent} if $\mathcal{I}^g$ is infinite for all $g\in\mathrm{FF}(\mathcal{I})$. An independent family $\mathcal{I}$ is \emph{maximal} if
$$\forall X\in [\omega]^\omega\;\exists g\in\mathrm{FF}(\mathcal{I})\text{ such that }\mathcal{I}^g\cap X\text{ or }\mathcal{I}^g\backslash X\text{ is finite,}$$

We will need a slight strengthening of maximality.

\begin{definition}
An independent family $\mathcal{I}$ is \emph{densely maximal} if $$\forall X\in [\omega]^\omega\text{ and }g'\in\mathrm{FF}(\mathcal{I})\;\exists g\supseteq g'\text{ in }\mathrm{FF}(\mathcal{I})\text{ such that }\mathcal{I}^g\cap X\text{ or }\mathcal{I}^g\backslash X\text{ is finite.}$$
\end{definition}

In other words, an independent family $\mathcal I$ is {\em densely maximal} if for each $X \in [\omega]^\omega$ the collection of $g$'s witnessing that $\mathcal I \cup \{X\}$ is not a larger independent family is dense in $(\mathrm{FF}(\mathcal I), \supseteq)$. 

\begin{definition}
Let $\mathcal I$ be an independent family. The \emph{density ideal of $\mathcal{I}$}, denoted $\mathrm{id}(\mathcal{I})$, is $$\{X\subseteq\omega\mid \forall g'\in\mathrm{FF}(\mathcal{I})\;\exists g\supseteq g'\text{ in }\mathrm{FF}(\mathcal{I})\text{ such that }\mathcal{I}^g\cap X\text{ is finite}\}.$$
Dual to the density ideal of $\mathcal{I}$ is the \emph{density filter of $\mathcal{I}$}, denoted $\mathrm{fil}(\mathcal{I})$ and defined as $$\{X\subseteq\omega\mid \forall g'\in\mathrm{FF}(\mathcal{I})\;\exists g\supseteq g'\text{ in }\mathrm{FF}(\mathcal{I})\text{ such that }\mathcal{I}^g\backslash X\text{ is finite}\}.$$
\end{definition}

Observe also that for an infinite independent family $\mathcal{I}$, none of the above definitions' meanings change if we replace the word ``finite'' with ``empty''.  We have as well the following from \cite[Lemma 5.4]{BFS21}.
\begin{lemma}\label{lemma0} A family $\mathcal{I}\subseteq [\omega]^\omega$ is densely maximal if and only if $$P(\omega)=\mathrm{fil}(\mathcal{I})\cup\langle\omega\backslash\mathcal{I}^g\mid g\in\mathrm{FF}(\mathcal{I})\rangle_{\mathrm{dn}}.$$
\end{lemma}
Where for a set $\mathcal X \subseteq [\omega]^\omega$ the set $\langle \mathcal X \rangle_{\mathrm{dn}}$ denotes the downward closure of $\mathcal X$ under $\subseteq^*$ i.e. $A \in \langle \mathcal X \rangle_{\mathrm{dn}}$ if and only if there is an $X \in \mathcal X$ with $A \subseteq^* X$. Later on we will similarly denote $\langle \mathcal X\rangle_{\mathrm{up}}$ for the upward closure of $\mathcal X$ under $\subseteq^*$. 

The following are easily verified, see \cite[Lemma 5.5]{BFS21}.
\begin{lemma}\label{lemma1}
\begin{enumerate}
\item
If $\mathcal I'$ is an independent family and $\mathcal{I}\subseteq\mathcal{I}'$ then $\mathrm{fil}(\mathcal{I})\subseteq\mathrm{fil}(\mathcal{I}')$;
\item if $\kappa$ is a regular uncountable cardinal and $\langle\mathcal{I}_\alpha\mid\alpha<\kappa\rangle$ is a continuous increasing chain of independent families then $\mathrm{fil}(\bigcup_{\alpha<\kappa}\mathcal{I}_\alpha)=\bigcup_{\alpha<\kappa}\mathrm{fil}(\mathcal{I}_\alpha)$;
\item If $\mathcal I$ is an independent family then $\mathrm{fil}(\mathcal{I})=\bigcup\{\,\mathrm{fil}(\mathcal{J})\mid \mathcal{J}\in [\mathcal{I}]^{\leq\omega}\}$.
\end{enumerate}
\end{lemma}

Recall that given a family $\mathcal F$ of subsets of $\omega$ we say that 
\begin{enumerate}
\item
$\mathcal F$ is a $P${\em -set} if every countable family $\{A_n \; | \; n < \omega\} \subseteq \mathcal F$ has a psuedointersection $B \in \mathcal F$,
\item
$\mathcal F$ is a $Q$-{\em set} if given every partition of $\omega$ into finite sets $\{I_n \; |\; n < \omega\}$ there is a {\em semiselector} $A \in \mathcal F$ i.e. $|A \cap I_n| \leq 1$ for all $n < \omega$,
\item
$\mathcal F$ is {\em Ramsey} if it is both a $P$-set and a $Q$-set.
\end{enumerate}
If $\mathcal F$ is a filter and a $P$-set (respectively a $Q$-set, Ramsey set) we call $\mathcal F$ a $P$-filter (respectively a $Q$-filter, Ramsey filter).

\begin{definition}
An independent family $\mathcal{I}$ is \emph{selective} if it is densely maximal and $\mathrm{fil}(\mathcal{I})$ is Ramsey.
\end{definition}

\begin{fact}[Shelah, see \cite{Sh92}] $\CH$ implies the existence of a selective independent family.
\end{fact}

Under certain circumstances countable support iterations of proper forcing notions over a model of $\CH$ will preserve that a given ground model selective independent family is maximal. Towards clarifying the meaning of ``certain conditions" recall the following preservation result, due to Shelah, see \cite[Lemma 3.2]{Sh92}.

\begin{theorem}\label{Shelah_preservation1}
Assume \textsf{CH}. Let $\delta$ be a limit ordinal and let $\langle\mathbb{P}_\alpha,\dot{\mathbb{Q}}_\alpha\mid\alpha<\delta\rangle$ be a countable support iteration of ${^\omega}\omega$-bounding proper posets. Let $\mathcal{F}\subseteq P(\omega)$ be a Ramsey set and let $\mathcal{H}$ be a subset of $P(\omega)\backslash\langle\mathcal{F}\rangle_{\mathrm{up}}$. If $V^{\mathbb{P}_\alpha}\vDash P(\omega)=\langle\mathcal{F}\rangle_{\mathrm{up}}\cup\langle\mathcal{H}\rangle_{\mathrm{dn}}$ for all $\alpha<\delta$ then $V^{\mathbb{P}_\delta}\vDash P(\omega)=\langle\mathcal{F}\rangle_{\mathrm{up}}\cup\langle\mathcal{H}\rangle_{\mathrm{dn}}$ as well.
\end{theorem}

We will also need the notion of {\em Cohen preserving}. 
\begin{definition}
Let $\P$ be a forcing notion. We say that $\P$ is {\em Cohen preserving} if every every new dense open subset of $2^{{<}\omega}$ (or, equivalently $\omega^{{<}\omega}$, ...) contains an old dense subset. More formally, $\P$ is Cohen preserving if for all $p \in \P$ and all $\P$-names $\dot{D}$ so that $p \forces$ ``$\dot{D} \subseteq2^{{<}\omega}$ is dense open" there is a dense $E \subseteq 2^{{<}\omega}$ in the ground model and a $q \leq_\P p$ so that $q \forces \check{E} \subseteq \dot{D}$.
\end{definition}

Being Cohen preserving is preserved by countable support iterations of proper forcing notions.

\begin{theorem}[Shelah, See Conclusion 2.15D, pg. 305 of \cite{PIP}, see also \cite{FM19}, Theorem 27]
If $\delta$ is an ordinal and $\langle \Q_\alpha, \dot{\mathbb R}_\alpha \; | \; \alpha < \delta\rangle$ is a countable support iteration of forcing notions so that for each $\alpha < \delta$ we have $\forces_\alpha$``$\dot{\mathbb R}_\alpha$ is proper and Cohen preserving" then $\Q_\delta$ is proper and Cohen preserving.
\label{Shelah_preservation2}
\end{theorem}

\section{$h$-Perfect Trees Preserve Selective Independent Families}
In this section we prove Main Theorem \ref{mainthm1}. First, recall the definition of $h$-perfect tree forcing from \cite{SMZnoCohen}.

\begin{definition}
Given a function $h:\omega \to \omega$ with $1 < h(n) < \omega$ for all $n < \omega$, the $h${\em -perfect tree forcing}, denoted $\mathbb{PT}_h$, is the forcing notion consisting of trees $p \subseteq \omega^{<\omega}$ so that the following hold:
\begin{enumerate}
\item
For all $t \in p$ and all $l \in {\rm dom}(t)$ we have $t(l) < h(l)$.
\item
Every $t \in p$ has either one or $h(l(t))$-many immediate successors in $T$. 
\item
For every $t \in p$ there is a $t ' \supseteq t$ with $t' \in p$ and there are $h(l(t'))$ many immediate successors of $t'$ in $p$.
\end{enumerate}
The order is inclusion. 
\label{hperfecttreedef}
\end{definition}

Note that the case where $h(n) = 2$ for all $n<\omega$ is simply Sacks forcing, while for a fast growing $h:\omega \to \omega$ this forcing will make the ground model reals measure zero so which $h$ we choose can affect the properties of the forcing significantly. This forcing notion was first considered in \cite{SMZnoCohen}. In \cite{SMZnoCohen} the following is shown.

\begin{fact}[\cite{SMZnoCohen}]
For any $h:\omega \to \omega$ with $1 < h(n) < \omega$ for all $n < \omega$ the following hold.
\begin{enumerate}
\item
$\mathbb{PT}_{h}$ is proper, and in fact satisfies Axiom A.
\item
$\mathbb{PT}_{h}$ is $\baire$-bounding.
\item
$\mathbb{PT}_{h}$ preserves $P$-points.
\end{enumerate}
\label{basicfacts}
\end{fact}

For the rest of this section fix an arbitrary function $h:\omega \to \omega$ so that $1 < h(n) < \omega$ for all $n < \omega$. We will prove first that forcing with $\mathbb{PT}_h$ preserves a ground model selective independent family. Then we will show that under $\CH$ arbitrary countable support iterations of $h$-perfect tree forcings (where the $h$ can change and need not even be in the ground model) will preserve selective independent families. First we introduce some arboreal terminology. If $p \in \mathbb{PT}_{h}$ and $n < \omega$ then a node $t \in p$ is an ${n}^{\rm th}${\rm -splitting node} if it has $h(l(t))$ many immediate successors and it has the $n - 1$  predecessors with this property. Denote by ${\rm Split}_n(p)$ the set of $n$-splitting nodes. We say that for two $h$-perfect trees $p, q \in\mathbb{PT}_{h}$ that $q \leq_n p$ if $q \leq p$ and for all $i < n + 1$ ${\rm Split}_i(p) ={\rm Split}_i(q)$. Given any $p \in \PT_h$ and any node $t \in p$ we let $p_t$ denote the tree $\{s \in p \; | \; s\subseteq t \, {\rm or} \, t \subseteq s\}$. Note that $p_t \in \PT_h$ and $p_t \leq p$ for any $t \in p$. 

\begin{lemma}
Let $p \in \PT_h$ and let $\dot{X}$ be a $\PT_h$-name for an infinite subset of $\omega$. There is a $q \leq p$ so that for all $n < \omega$ and any $n$-splitting node $t \in q$ we have that $q_t$ decides $\dot{X} \cap \check{n}$. 
\end{lemma}

Such a $q$ is called {\em preprocessed for} $\dot{X}$.

\begin{proof}
This is a standard fusion argument but we sketch it for completeness. Fix $p$ and $\dot{X}$ as in the lemma. Inductively define a fusion sequence $...\leq_n p_n \leq_{n-1} p_{n-1} \leq_{n-2} ... \leq_1 p_1 \leq_0 p_0 = p$ as follows. Given $k < \omega$ and $p_k$, for each $k^{\rm th}$ splitting node $t \in {\rm Split}_k(p_k)$ let $p_t' \leq(p_k)_t$ decide $\dot{X} \cap \check{k}$. Let $p_{k+1} = \bigcup_{t \in {\rm Split}_k(p_k)} p'_t$. Clearly $p_{k+1} \leq_k p_k$ and the fusion $q:= \bigcap_{k < \omega} p_k$ is as needed.
\end{proof}

We now move to the first substantial lemma.

\begin{lemma}
$\PT_h$ is Cohen preserving.
\label{cohenpreserving}
\end{lemma}

\begin{proof}
Fix $p \in \PT_h$ and suppose that $\dot{D}$ is a $\PT_h$-name for a dense open subset of $2^{{<}\omega}$. Enumerate $2^{{<}\omega}$ as $\{s_n \; | \; n<\omega\}$. We will inductively construct sequences $\{p_n\; | \; n < \omega\}$ and $\{t_n\; | \; n < \omega\}$ so that the following hold.
\begin{enumerate}
\item
$p_0 = p$.
\item
$p_{n+1} \leq_n p_n$ for all $n < \omega$.
\item 
$t_n \supseteq s_n$ for all $n < \omega$.
\item 
$p_{n+1} \forces \check{t}_n \in \dot{D}$.
\end{enumerate}

Given such sequences, let $q = \bigcap_{n<\omega} p_n$ and $E = \{t_n \; | \; n < \omega\}$. Clearly $q \forces \check{E} \subseteq \dot{D}$ and $E$ is dense so, assuming we can construct such sequences we will be done. This is done by induction. Given $p_k$, enumerate ${\rm Split}_k(p_k)$ as $\{u_i \; | \; i <l\}$. Note that $l \in \omega$ depends not just on $h$ but also on $p_k$ but what matters here is that it is finite (which it is). Now let $p'_{k, 0} \leq (p_k)_{u_0}$ decide some $t^0_k \supseteq s_k$ to be in $\dot{D}$ (since $\dot{D}$ is forced to be dense this is possible). Next, let $p'_{k, 1} \leq (p_k)_{u_1}$ decide some $t^1_k \supseteq t^0_k$ to be in $\dot{D}$. Continuing this way, inductively,let  for all $0 < i < l$ let $p'_{k, i} \leq (p_k)_{u_i}$ decide some $t^i_k \supseteq t^{i-1}_k$ to be in $\dot{D}$. Let $p_{k+1} = \bigcup_{i < l} p'_{k, i}$ and $t_{k+1}$ be $t^{l-1}_k$. Since $\dot{D}$ is forced to be open we have that $p_{k+1} \forces \check{t}_{k+1} \in \dot{D}$ so we're done.
\end{proof}

Fix a selective independent family $\mathcal I$ in the ground model. 
\begin{lemma}
If $G \subseteq \PT_h$ is generic over $V$ then in $V[G]$ the ideal $\mathrm{id}(\mathcal I)$ is generated by $\mathrm{id}(\mathcal I) \cap V$. 
\label{idealpreserving}
\end{lemma}
Note that dually this lemma implies that in $V[G]$ the filter $\mathrm{fil}(\mathcal I)$ is generated by $\mathrm{fil}(\mathcal I) \cap V$.

\begin{proof}
Let $p \in \PT_h$, let $\dot{X}$ be a $\PT_h$-name for an infinite subset of $\omega$ and suppose that $p \forces \dot{X} \in {\rm id}(\mathcal I)$. We need to find a ground model $Y \in [\omega]^\omega$ and an $r \leq p$ so that $r \forces \dot{X} \subseteq \check{Y}$. Towards this, via a fusion argument, or just using the properness of $\PT_h$, find a $q \leq p$ and a countable $\mathcal J \subseteq \mathcal I$ so that $q \forces \dot{X} \in {\rm id}(\mathcal J)$. Since $\mathcal J$ is countable, we can associate $\mathrm{FF}(\mathcal J)$ with $2^{<\omega}$. Let $\dot{D}$ be the $\PT_h$-name for the dense open subset of $2^{<\omega}$ defined by $q \forces \check{g} \in \dot{D}$ if and only if $\mathcal J^g \cap \dot{X} = \emptyset$. Since $\PT_h$ is Cohen preserving there is an $r \leq q$ and a dense $E \subseteq 2^{<\omega}$ so that $r \forces \check{E} \subseteq \dot{D}$. Let $Y = \bigcap_{h \in E} (\omega \setminus \mathcal J^h)$. Observe that $Y \in {\rm id}(\mathcal J)$ and hence $Y \in {\rm id}(\mathcal I)$. To see this, let $g' \in {\rm FF}(\mathcal J)$ be arbitrary and let $g \supseteq g'$ be in $E$. We have that $Y \subseteq \omega \setminus \mathcal J^g$ and hence $Y \cap \mathcal J^g = \emptyset$ which by definition means that $Y$ is in the ideal. The following claim now completes the proof.

\begin{claim}
$r \forces \dot{X} \subseteq \check{Y}$. 
\end{claim}

\begin{proof}
Let $r \in G$ be $\PT_h$-generic over $V$. Note that by the way $\dot{D}^G$ is defined we have that $\dot{X} = \bigcap_{g \in \dot{D}^G} (\omega \setminus \mathcal B^g)$ and since $E \subseteq \dot{D}^G$ we're done.
\end{proof}
\end{proof}

\begin{theorem}[$\CH$]
If $\mathcal I$ is a selective independent family and $G \subseteq \PT_h$ is generic over $V$ then $V[G] \models$``$\mathcal I$ is a selective independent family".
\label{preserveI}
\end{theorem}

\begin{proof}
There are three things to check. We need to show 1) that ${\rm fil}(\mathcal I)$ is a $P$-filter, 2) that ${\rm fil}(\mathcal I)$ is a $Q$-filter and 3) that $\mathcal I$ is densely maximal. We take these one at a time. First though, since we assume $\CH$ in $V$, note that the fact that ${\rm fil}(\mathcal I)$ is a $P$-set implies that we can assume that ${\rm fil}(\mathcal I)$ is generated by an $\omega_1$ length $\subseteq^*$-descending sequence $\{B_\alpha \; | \; \alpha < \omega_1\}$. In other words for all $\alpha < \beta < \omega_1$ we have $B_\beta \subseteq^* B_\alpha$ and every $A \in {\rm fil}(\mathcal I)$ is almost contained in some (equivalently a tail of) $B_\gamma$. Fix such a sequence $\{B_\alpha \; | \; \alpha < \omega_1\}$. 

To see that ${\rm fil}(\mathcal I)$ is a $P$-filter in $V[G]$ then, note that if $\{A_n \; | \; n < \omega\} \subseteq {\rm fil}(\mathcal I)$ in $V[G]$ by Lemma \ref{idealpreserving} there are countable ordinals $\{\gamma_n\; |\; n < \omega\}$ so that for all $n < \omega$ we have $B_{\gamma_n} \subseteq^* A_n$. Let $\gamma \geq {\rm sup}_{n < \omega} \gamma_n$. We have $B_\gamma \subseteq^* A_n$ for all $n < \omega$ so ${\rm fil}(\mathcal I)$ is a $P$-filter in $V[G]$.

The fact that ${\rm fil}(\mathcal I)$ is a $Q$-filter still in $V[G]$ follows immediately from the fact that $\PT_h$ is $\baire$-bounding.

Thus it remains to see that $\mathcal I$ remains densely maximal in $V[G]$. Suppose not and let $\dot{X}$ be a $\PT_h$-name for an infinite subset of $\omega$ so that in $V[G]$ $\dot{X}^G$ is not in ${\rm fil}(\mathcal I)$ and for all $g \in {\rm FF}(\mathcal I)$ we have $\dot{X} \nsubseteq \omega \setminus \mathcal I^g$. Let $p \in G$ force this. Without loss of generality we may assume that $p$ is preprocessed for $\dot{X}$ i.e. for each $n < \omega$ and every $t\in {\rm Split}_n(p)$ we have that $p_t$ decides $\dot{X} \cap \check{n}$. For each split node $t$ of $p$ let $Y_t = \{n \; | \; p_t \nVdash \check{n} \notin \dot{X}\}$. Note that for all split nodes $t$ we have that $p_t \forces \dot{X} \subseteq \check{Y}_t$. 

\begin{claim}
For all split nodes $t$ of $p$ we have $Y_t \in {\rm fil}(\mathcal I)$.
\end{claim}

\begin{proof}
Otherwise there is a $t \in {\rm Split}(p)$ so that $Y_t \subseteq \omega \setminus \mathcal I^g$ for some $g \in {\rm FF}(\mathcal I)$ but since $p_t \forces \dot{X} \subseteq \check{Y}_t$ we have that $p_t \forces \dot{X} \subseteq \omega \setminus \mathcal I^g$. However this contradicts the choice of $p$.
\end{proof}

Since ${\rm fil}(\mathcal I)$ is a $P$-filter generated by ground model elements there is a $C \in {\rm fil}(\mathcal I) \cap V$ so that $C \subseteq^* Y_t$ for all $t \in {\rm Split}(p)$. Let $f \in \baire$ be a strictly increasing function such that for all $n < \omega$ we have $C \setminus f(n) \subseteq \bigcap \{Y_t \; | \; t \in {\rm Split}_j(p), \; j \leq n + 2\}$. The following is proved in \cite{Sh92}, as well as \cite[Lemma 3.15]{CFGS21} but we include it for completeness.

\begin{claim}
There is a $C^* \subseteq C$ so that $C^* \in {\rm fil}(\mathcal I) \cap V$ and, letting $\{k_n \; | \; n < \omega\}$ be a strictly increasing enumerate of $C^*$, we have that $f(k_n) < k_{n+1}$ for all $n < \omega$ and $f(1) < k_1$.
\end{claim}

\begin{proof}
This follows from the fact that ${\rm fil}(\mathcal I)$ is a $Q$-filter\footnote{In fact a family $\mathcal F \subseteq [\omega]^\omega$ is a $Q$-filter if and only if for each increasing $f \in \baire$ there is a $C^* = \{k_n \; | \; n < \omega\} \in \mathcal F$ such that $f(k_n) < k_{n+1}$, see \cite[Lemma 3.15]{CFGS21}.}. Inductively find a sequence $\{n_l\}_{l \in \omega}$ so that $n_0 = 0$ and $$n_{l+1} = {\rm min}\{n \; | \; n_l < n {\rm \, and \, for \, all} \, m < n_l \, f(m) < n\}$$. Consider the interval partition $\mathcal E_0 = \{[n_{3l}, n_{3l + 3})\}_{l \in \omega}$. Since ${\rm fil}(\mathcal I)$ is a $Q$-set and a filter there is a $C_1 \subseteq C$ so that for all $l < \omega$ we have $|C_1 \cap [n_{3l}, n_{3l + 3})| \leq 1$. Now define an equivalence relation $\mathcal E_1$ on $\omega$ by $$m \equiv_{\mathcal E_1} k \, {\rm iff} \, m = k \lor m, k \in C_1 \land (m < k \leq f(m) \lor k < m \leq f(k)).$$
In words, this says that every element of $\omega \setminus C_1$ is in their own equivalence class and distinct elements $m, k \in C_1$ are $\mathcal E_1$-equivalent just in case applying $f$ to the smaller one is greater or equal to the bigger one. Every $\mathcal E_1$ equivalence class has at most two members. To see this, suppose $ m_1 < m_2 < m_3 \in C_1$ were all in the same equivalence class. By definition of $\mathcal E_1$ we have $m_1 <  m_2 < m_3 \leq f(m_1)$. However, since $C_1$ is a semiselector for the interval partition $\mathcal E_0$ there are distinct $l_1 < l_2 < l_3$ so that for all $i \in \{1, 2, 3\}$ we have $m_i \in [n_{3l_i}, n_{3l_{i+1}})$. Thus we get $m_1 < n_{3l_2} \leq m_2  < n_{3l_3} \leq m_3 \leq f(m_1)$ but by the definition of the $n_l$ sequence we also have $f(m_1) \leq n_{3l_2 + 1} < n_{3l_3}$ which is a contradiction. 

Now let $C_2 \subseteq C_1$ be a semiselector for $\mathcal E_1$ in ${\rm fil}(\mathcal I)$. Without loss of generality $0 \in C_2$. Let $\{k_n\}$ be an increasing enumeration of $C_2$. For all $n < n'$ we have that $n$ and $n'$ are not in the same $\mathcal E_1$ equivalence class and therefore $f(n) < n'$. As such $C_2 = C^*$ is as needed.

For the final point note that ${\rm fil}(\mathcal I)$ is closed under finite changes to elements so we can augment $C^*$ to get $f(1) < k_1$ as needed.
\end{proof}

We will find a $q \leq p$ forcing that $C^* \subseteq \dot{X}$ which contradicts the choice of $X$. Obviously this will be a fusion argument. Let $t^* = {\rm Stem}(p)$ and let $p_0 = p = p_{t^*}$. Now for each $i \in h(l(t^*))$ let $w(t^*, i) \in p_0$ be a $1$-splitting node extending $(t^* )^\frown i$. Since $k_1 > f(1)$ we have that $k_1 \in \bigcap \{Y_{w(t^*, i)} \; | \; i \in h(l(t^*))\}$. This means that for each $i \in h(l(t^*))$ there is a $w'(t^*, i) \in {\rm Split}_{k_1 + 1}(p_0)$ extending $w(t^*, i)$ which forces that $k_1 \in \dot{X}$ since $p_0$ is preprocessed. Let $p_1 = \bigcup_{i \in h(l(t^*))} p_{w'(t^*, i)}$. Note that $p_1 \leq_0 p_0$ and forces that $k_1 \in \dot{X}$. Also note that ${\rm Split}_1(p_1) \subseteq {\rm Split}_{k_1 + 1}(p_0)$ and, by construction, for each $m$ we have that ${\rm Split}_{m}(p_1) \subseteq {\rm Split}_{k_1 +  m}(p_0)$

Now proceed inductively defining $p_{n+1}$ as follows. Assume $p_n$ has been defined and that for all $m < \omega$ we have ${\rm Split}_{n +m}(p_n) \subseteq {\rm Split}_{k_n + m} (p_0)$. Observe that we have $k_{n+1} \in \bigcap\{Y_t \; | \; t \in {\rm Split}_{n}(p_n)$ since $k_{n+1} > f(k_n)$ and hence we can find for each $t \in {\rm Split}_n(p_n)$ and each $i \in h(l(t))$ a $w(t, i) \in {\rm Split}_{k_{n+1} + 1}(p_0)$ in $p_n$ which contains $t^\frown i$ so that $(p_n)_{w(t, i)} \forces \check{k}_{n+1} \in \dot{X}$. Let $p_{n+1} = \bigcup_{t \in {\rm Split}_n(p_n)} \bigcup_{i \in h(l(t))} (p_n)_{w(t, i)}$. Clearly this is as needed. 

Let $q$ be the fusion of the $p_n$'s. We have that $q \forces \check{C}^* \subseteq \dot{X}$ contradicting the fact that $\dot{X}$ is forced not to be in ${\rm fil}(\mathcal I)$ so we're done.
\end{proof}

Now we show how to lift the above proof to show that iterations of $h$-perfect tree forcing notions preserve selective independent families.
 
\begin{theorem}[$\CH$]
Let $\delta$ be an ordinal and $\mathcal I$ be a selective independent family. Let $\langle \P_\alpha, \dot{\Q}_\alpha \; | \; \alpha < \delta\rangle$ be a countable support iteration of posets so that for all $\alpha < \delta$ we have $\forces_\alpha$``$\dot{\Q}_\alpha$ is $\PT_h$ for some $h \in \baire$ with $1 < h(n) < \omega$ for all $n < \omega$". If $G \subseteq \P_\delta$ is generic over $V$ then $V[G] \models$ ``$\mathcal I$ is a selective independent family".
\label{iteration}
\end{theorem}

\begin{proof}
The proof is by induction on $\delta$. Note first that by Theorems \ref{Shelah_preservation1} and \ref{Shelah_preservation2} and Lemmas \ref{cohenpreserving} and \ref{idealpreserving} we have that for each $\alpha \leq \delta$ that $\P_\alpha$ is Cohen preserving and forces that ${\rm fil}(\mathcal I)$ is generated by ground model sets. This guarantees that $\P_\alpha$ forces that ${\rm fil}(\mathcal I)$ is a $P$-filter. Moreover being $\baire$-bounding ensures that ${\rm fil}(\mathcal I)$ is a $Q$-filter hence ${\rm fil}(\mathcal I)$ is forced to be Ramsey by every $\P_\alpha$ for $\alpha \leq \delta$. Therefore we just need to ensure that $\forces_\delta$``$\mathcal I$ is densely maximal" under the assumption that for all $\alpha < \delta$ we have that $\forces_\alpha$``$\mathcal I$ is densely maximal". To show this we will use the characterization of dense maximality given by Lemma \ref{lemma0}. We now consider two separate cases:

\noindent \underline{Case 1}: $\delta = \beta + 1$ for some $\beta$. The proof of this case is almost verbatim the same as the proof of theorem \ref{preserveI} noting that, by the above, we can assume that ${\rm fil}(\mathcal I)^{V^{\P_\beta}}$ is generated by ${\rm fil}(\mathcal I) \cap V$ and the $f$ and $C^*$ found in that proof can be assumed to come from the ground model by $\baire$-boundedness.

\noindent \underline{Case 2}: $\delta$ is a limit ordinal. Inductively we have that if $\beta < \delta$ and $G_\beta \subseteq \P_\beta$ is generic over $V$ then $$V[G_\beta] \models P(\omega) = \langle {\rm fil}(\mathcal I) \cap V\rangle_{\rm up} \cup \langle \omega \setminus \mathcal I^g\; | \; g \in \mathsf{FF}(\mathcal I)\rangle_{\rm dn}.$$ But then by Theorem \ref{Shelah_preservation1} plus the fact that ${\rm fil}(\mathcal I)$ is a Ramsey filter in $V^{\P_\delta}$ we get $$\forces_\delta P(\omega) = \langle {\rm fil}(\mathcal I) \cap V\rangle_{\rm up} \cup \langle \omega \setminus \mathcal I^g\; | \; g \in \mathsf{FF}(\mathcal I)\rangle_{\rm dn}.$$
Which, by Lemma \ref{lemma0} is exactly what we needed to show.
\end{proof}

\section{Applications}
We now turn to applications of the results from the previous section. The most obvious of these is that there is a small independent family in any model obtained by iteratively forcing with $h$-perfect tree partial orders. In particular we get the following as a corollary to Theorem \ref{iteration}.
\begin{corollary}
It is consistent that $\mfi = \aleph_1 < {\rm non}(\mathcal N) = \aleph_2$.
\end{corollary}

As mentioned in the introduction this consistent inequality was first shown in \cite[Theorem 3.8]{BHHH04}, though by a very different construction. 

In the models constructed in \cite{SMZnoCohen} many interesting properties hold with regards to the structure of the strong measure zero sets, for example the consistency of ``the additivity of the strong measure zero ideal is $\aleph_2 = 2^{\aleph_0}$". As a consequence all of these are also consistent with $\mfi = \aleph_1$. One thing to note as a consequence of this is the following.

\begin{corollary}
It is independent of $\ZFC$ whether there is a maximal independent family of strong measure zero.
\end{corollary}

\begin{proof}
In the Laver model, \cite[Model 7.6.13]{BarJu95} every strong measure zero set is countable so no maximal independent family has strong measure zero. By contrast if the additivity of the strong measure zero ideal is $\aleph_2$ then in particular any set of reals of size $\aleph_1$ will be strong measure zero, in particular any selective independent family from the ground model.

\end{proof}

We can also now iteration mixings of $h$-perfect posets with other proper partial orders which iteratively preserve small selective independent families.

\begin{theorem}
The following are consistent.
\begin{enumerate}
\item
$\mfi = \mfu < \non (\mathcal N)$
\item
$\mfi < \mfu = \non(\mathcal N)$
\end{enumerate}
\label{inequalities}
\end{theorem}

\begin{proof}
For the first inequality, as noted above in Fact \ref{basicfacts}, for any $h$ we have that $\mathbb{PT}_h$ preserves $P$-points hence in any model constructed by iterating $h$-perfect tree forcings with countable support over a model of $\CH$ there will be a $P$-point base of size $\aleph_1$. For the second inequality alternating between $\PT_h$ for e.g. $h(n)= 2^n$ and the forcing notions $\Q_\mathcal I$ of \cite{Sh92} (alongside some bookkeeping device) will increase $\mfu$ but preserves selective independent families.
\end{proof}

Finally we note some applications to definability. 

\begin{theorem}
Both inequalities featured in Theorem \ref{inequalities} are consistent with a $\Pi^1_1$ independent family of size $\aleph_1$, a $\Delta^1_3$ well order of the reals and, in the case of the first inequality, a $\Pi^1_1$ ultrafilter base for a $P$-point of size $\aleph_1$.
\end{theorem}

\begin{proof}
Schilhan \cite{Schilhanultrafilter} has shown that in $L$ there is a $\Pi^1_1$ ultafilter base for a $P$-point and Brendle, Fischer and Khomskii \cite{DefMIF} have shown that in $L$ there is a $\Sigma^1_2$ selective independent family and that if there is a $\Sigma^1_2$ maximal independent family then there is a $\Pi^1_1$ maximal independent family. It follows that all of these objects can be preserved by the iterations described in the proof of Theorem \ref{inequalities} assuming the ground model is $L$. Finally for the $\Delta^1_3$-well order of the reals we apply the forcing from \cite{BFS21} noting that the main theorem of that paper is precisely that such objects can be preserved by this forcing, even when other forcing notions, such as $\mathbb{PT}_h$ are added to the iteration.
\end{proof}

\section{Conclusion and Open Questions}
The proof of Main Theorem \ref{mainthm1} are almost verbatim the same as those for Sacks forcing \cite{FM19}, Shelah's forcing for killing a maximal ideal used in \cite{Sh92} and very similar to the proof for the coding with perfect tree forcing from \cite{BFS21}. In particular really only structural properties of the forcing are used. This suggests there should be a general property of proper, $\baire$-bounding forcing notions which imply that small selective independent families are preserved. The following seems like the first place to go to isolate such a property.

\begin{question}
Suppose $\delta$ is an ordinal and $\langle \P_\alpha , \dot{\Q}_\alpha \; | \; \alpha < \delta\rangle$ is a countable support iteration of proper, $\baire$-bounding, Cohen preserving forcing notions. Let $\mathcal I$ is a selective independent family in $V \models \CH$. If for all $\alpha < \delta$ $\forces_\alpha$`` $\dot{\Q}_\alpha$ forces ``${\rm fil}(\mathcal I)$ is generated by ground model sets" then does $\P_\delta$ preserve the maximality of $\mathcal I$?
\end{question}

\end{document}